\documentclass[10pt]{article}
\usepackage{amsmath,amsfonts,amsthm,color}
\usepackage{enumerate}
\usepackage{algorithm,algorithmic}
\usepackage{fullpage}

\newtheorem{theorem}{Theorem}

\newtheorem{corollary}{Corollary}

\newtheorem{proposition}{Proposition}
\newtheorem{problem}{Problem}

\theoremstyle{definition}
\newtheorem{definition}{Definition}
\newtheorem{remark}{Remark}
\theoremstyle{plain}

\newcommand{\calF}{\mathcal{F}}

\newcommand{\ignore}[1]{}

\title{Semi-Strong Coloring of Intersecting Hypergraphs}
\author{
Eric Blais\thanks{
Computer Science Department, 
Carnegie Mellon University,
Pittsburgh, PA 15213, USA. 
Email: \texttt{eblais@cs.cmu.edu}.}
\and
Amit Weinstein\thanks{
Blavatnik School of Computer Science, Tel Aviv University, Tel Aviv
69978, Israel. Email: \texttt{amitw@tau.ac.il}. Research supported in part by
an ERC Advanced grant and by the Israeli Centers of Research
Excellence (I-CORE) program.}
\and
Yuichi Yoshida\thanks{
  School of Informatics, Kyoto University, Kyoto 606-8501, Japan, and
  Preferred Infrastructure, Inc., Tokyo 113-0033, Japan.
  Email: \texttt{yyoshida@kuis.kyoto-u.ac.jp}
}
}

\begin{document}
\maketitle

\begin{abstract}
For any $c \ge 2$, a \emph{$c$-strong coloring} of the hypergraph $G$ is an 
assignment of colors to the vertices of $G$ such that for every edge $e$ of $G$,
the vertices of $e$ are colored by at least $\min\{c,|e|\}$ distinct colors.
The hypergraph $G$ is \emph{$t$-intersecting} if every two edges of $G$
have at least $t$ vertices in common. 
We ask: for fixed $c \ge 2$ and $t \ge 1$,
what is the minimum number of colors that is sufficient to $c$-strong color
any $t$-intersecting hypergraphs?

The purpose of this note is to answer the question for some values of $t$ and $c$ and,
more importantly, to describe the settings for which the question is still open.  
We show that when $t \le c-2$, no finite number of colors is sufficient to $c$-strong
color all $t$-intersecting hypergraphs. It is still unknown whether a
 finite number of colors suffices for the same task when $t = c-1$ and $c > 2$.
In the last case, when $t \ge c$, we show with a probabilistic argument
that a finite number of colors is sufficient to $c$-strong
color all $t$-intersecting hypergraphs, but
a large gap still remains between the best upper 
and lower bounds on this number.
\end{abstract}

\section{Introduction}
The problem of coloring graphs and hypergraphs has a long and rich history (see, 
e.g.,~\cite{JT94,MR01}). In the case of graphs, the notion of vertex coloring has a single 
natural definition: an assignment of labels to the vertices of a graph is a \emph{proper coloring}
if the endpoints of any edge in the graph are assigned distinct labels.
For hypergraphs, however, there exist different natural definitions of vertex coloring.
The most common definition, also called \emph{weak} coloring, is an 
assignment of colors to the vertices such that no edge is monochromatic. Another common 
definition, called \emph{strong} coloring, is an assignment of colors to the vertices such that 
all the vertices contained in an edge have distinct colors. 

There is a more general notion of hypergraph vertex coloring that encompasses both the weak and strong coloring definitions. We call this notion \emph{semi-strong} coloring.

\begin{definition}[Semi-strong coloring]
For a fixed $c \ge 2$, a \emph{$c$-strong coloring} of the hypergraph $G$ is an assignment of colors
to its vertices such that each edge $e$ of $G$ covers vertices with at least $\min\{c, |e|\}$ distinct colors.
The \emph{$c$-strong chromatic number} of $G$, denoted $\chi(G, c)$, is the minimum number of
colors required to $c$-strong color $G$.
\end{definition}

The definition of weak coloring corresponds to that of 2-strong coloring, and the definition of strong coloring is equivalent to $\infty$-strong coloring.\footnote{More generally, the notion of $c$-strong coloring of a hypergraph $G$ is equivalent to the strong coloring of $G$ whenever $c$ is at least as large as the cardinality of the largest edge in $G$.}

\medskip

The main focus of this note is the semi-strong coloring of intersecting hypergraphs. A hypergraph is 
\emph{$t$-intersecting} if the intersection of any two of its edges contains at least $t$ vertices. 
The set of edges of a $t$-intersecting hypergraph is often referred to as a $t$-intersecting family. 
Our goal is to determine the minimum number of colors that are sufficient to $c$-strong color any
$t$-intersecting hypergraph.

\begin{definition}[Chromatic number of intersecting hypergraphs]
Given two integers $c \geq 2$ and $t \geq 0$, the \emph{$c$-strong chromatic number of $t$-intersecting hypergraphs}, denoted $\chi(t,c)$, is the minimum number of colors which suffices to $c$-strong color $\emph{any}$ $t$-intersecting hypergraph.
\end{definition}

With this notation, our goal can be restated as follows: determine $\chi(t,c)$ for every $t \ge 0$ and every $c \ge 2$.
A first step toward that goal is to determine when $\chi(t,c)$ is finite and when it is unbounded. As we will see below, $\chi(t,c)$ is finite whenever $t \ge c$ and it is unbounded whenever $t \le c-2$.  When $c = 2$ and $t = c-1 = 1$, we also know that $\chi(1,2)$ is finite.  The remaining cases remain open.

\begin{problem}
\label{prob:finite}
Determine whether $\chi(c-1,c)$ is finite or not for every $c > 2$.
\end{problem}

In particular, the problem of determining whether $\chi(2,3)$ is finite or not is still open. In other words, is there an upper bound on the number of colors required to $3$-strong color $2$-intersecting hypergraphs?

\medskip

When $\chi(t,c)$ is finite, a natural question is to determine its precise value. As we see in the next section, this goal is easily accomplished in the case where $c = 2$.  For every $c > 2$, however, the picture is much less clear.  In Section~\ref{sec:lb}, we show that for every $t \ge c \ge 2$, we have the lower bound $\chi(t,c) \ge 2(c-1)$.  It seems reasonable to believe that this lower bound is tight. The best upper bound for the same chromatic numbers, however, is far from tight. 
We thus have the following open problem.

\begin{problem}
\label{prob:cc}
Determine whether $\chi(c,c) = 2(c - 1)$ for every $c > 2$.
\end{problem}

For any $t' > t$, the inequality $\chi(t',c) \leq \chi(t,c)$ follows immediately from the observation that $t'$-intersecting hypergraphs are also $t$-intersecting.  A positive answer to Problem~\ref{prob:cc} would therefore immediately imply that $\chi(t,c) = 2(c-1)$ for every $t \ge c$.  It might be easier to first determine whether $\chi(t,c) = 2(c-1)$ for values of $t$ that are much greater than $c$.  But even the problem of determining whether the limit of $\chi(t,c)$ as $t \to \infty$ equals $2(c-1)$ is open.

\begin{problem}
For every $c > 2$, determine whether $\lim_{t \to \infty} \chi(t,c) = 2(c-1)$.
\end{problem}

For the last problem we return to the chromatic number $\chi(c-1,c)$. If it is finite, can we determine its exact value?  When $c = 2$, we know that $\chi(1,2) = 3 = 2c-1$.  For larger values of $c$, we have the matching lower bound $\chi(c-1,c) \ge 2c-1$. The final problem asks whether this bound is tight.

\begin{problem}
For every $c > 2$, determine whether $\chi(c-1,c) = 2c - 1$.
\end{problem}

In the rest of this note, we present the known results on the chromatic numbers of intersecting hypergraphs.  Section~\ref{sec:weak} presents the folklore results that establish the exact value of $\chi(t,2)$ for every $t \ge 0$.  Section~\ref{sec:lb} establishes lower bounds on the values of $\chi(t,c)$ for every $t \ge 0$, and Section~\ref{sec:ub} introduces the probabilistic argument for obtaining upper bounds on $\chi(t,c)$ when $t \ge c-1$.

\section{Results for weak coloring}
\label{sec:weak}

When we restrict our attention to the $2$-strong (i.e., weak) coloring of hypergraphs, we have a complete characterization of the chromatic number $\chi(t,2)$ for every $t \ge 0$.

\begin{proposition}%[Folklore]
\label{prop:weak}
The $2$-strong chromatic numbers for $t$-intersecting hypergraphs are
\begin{align*}
\chi(0, 2) &= \infty\ , \\
\chi(1, 2) &= 3\ \mbox{, and} \\
\chi(t, 2) &= 2 \mbox{ for every $t \geq 2$.}
\end{align*}
\end{proposition}

\begin{proof}
The case where $t=0$ is trivial: since every hypergraph is $0$-intersecting, no finite number of colors suffices to color all $0$-intersecting hypergraphs. 

Consider now the case where $t=1$. The graph $G = ([3], {[3] \choose 2})$ is $1$-intersecting and
has chromatic number $\chi(G,2) = 3$, so $\chi(1,2) \ge 3$.  To establish a matching upper bound, fix
$G$ to be any $1$-intersecting hypergraph whose edges all have size at least $2$.  Let $e$ be any edge
of $G$ that does not strictly contain any other edge $e'$ of $G$. Color the vertices of $e$ with $2$ colors
in any arbitrary way,\footnote{Except, of course, for the degenerate colorings that apply only one of the two colors to all the vertices in the edge.} and color the remaining vertices of $G$ with a third color.  The resulting coloring is valid: $e$ covers vertices with two different colors and every other edge contains one vertex in $e$ and a vertex outside of $e$.

We are left with the case $t\ge2$. Notice that $\chi(t, c)$ is monotone non-increasing with $t$ and therefore
showing that $\chi(2,2)=2$ completes the proof (as 2 colors are obviously required). We again provide
a simple coloring algorithm given a 2-intersecting hypergraph $G$. Fix an arbitrary order on the vertices of $G$. Color the vertices one by one according to this order. In this process, assign the current vertex the color blue unless doing so causes an edge to become fully colored monochromatically; 
in that case color the vertex red. We claim that this process succeeds in generating a valid coloring of $G$. To prove the claim, assume on the contrary that the process causes some edge $e$ to be monochromatic. Let $v$ be the last vertex in $e$ that was colored in the process.  The vertex $v$ must have been colored red, since by definition if the rest of $e$ contained all blue vertices, we would not color $v$ blue as well. So all the remaining vertices covered by $e$ are red.  But if we colored $v$ red, it must be because there is some other edge $e'$ in $G$ for which $v$ is the last vertex to be colored and where all other vertices covered by $e'$ have been colored blue. But since $G$ is $2$-intersecting, $e$ and $e'$ must both cover some other vertex $v' \neq v$. This is a contradiction, since of course $v'$ cannot be colored both red and blue.
\end{proof}

\section{General lower bounds}
\label{sec:lb}

As we have mentioned in the introduction, the trivial observation that $(t+1)$-intersecting hypergraphs are also $t$-intersecting implies that the $c$-strong chromatic number of $t$-intersecting hypergraphs is non-increasing in $t$. In other words, for any $c \ge 2$ and any $t \ge 0$, we have $\chi(t+1,c) \le \chi(t,c)$.  
The following proposition shows that the semi-strong chromatic number of intersecting hypergraphs satisfies a different monotonicity property when we increase both $t$ and $c$.

\begin{proposition}
\label{prop:increasing}
For any $c \ge 2$ and any $t \ge 0$,
we have
$
\chi(t+1, c+1) \geq \chi(t, c) + 1.
$
\end{proposition}

\begin{proof}
Let $G$ be a $t$-intersecting hypergraph with $c$-strong chromatic number $\chi(G,c) = \chi(t,c)$.
Define $G'$ to be the $(t+1)$-intersecting hypergraph obtained by adding a new vertex $v$ and including it in each of
the edges of $G$.  Since $\chi(t+1, c+1) \geq \chi(G', c+1)$, it suffices to show that $\chi(G', c+1) \ge \chi(G, c)+1 = \chi(t, c)+1$.

Consider any $(c+1)$-strong coloring of $G'$ that uses  $\ell$ colors.  For each edge $e \cup \{v\}$ of $G'$,
this coloring must assign at least $\min\{c+1,|e|+1\}$ distinct colors to the vertices covered by this edge. This implies
that the vertices in the edge $e$ (without the new vertex $v$) must be colored by $\min\{c,|e|\}$ distinct colors
that are all different from the color assigned to $v$.  Since this is true for any edge $e$ of $G$, we obtain a 
$c$-strong coloring of $G$ with $\ell-1$ colors by arbitrarily recoloring any vertex of $G$ that received the same 
color as $v$. Therefore, $\chi(G', c+1) \ge \chi(G, c)+1$, as we wanted to show.
\end{proof}

Proposition~\ref{prop:increasing} immediately implies that $\chi(t,c)$ is unbounded
whenever $t \le c - 2$.

\begin{corollary}
For any $c \geq 2$ and any $t \leq c-2$, we have $\chi(t, c) = \infty$.
\end{corollary}

\begin{proof}
Applying Proposition~\ref{prop:increasing} a total of $t$ times, we obtain
$$
\chi(t,c) \ge \chi(t-1,c-1) \ge \chi(t-2,c-2) \ge \cdots \ge \chi(0,c-t) \ .
$$
But when $c-t \ge 2$, no finite number of colors is sufficient to $(c-t)$-strong color all $0$-intersecting hypergraphs 
since this class includes all hypergraphs.
\end{proof}

The following two propositions give the lower bounds on $\chi(t,c)$ when $t \ge c-1$.

\begin{proposition}
For any $c \geq 2$, we have $\chi(c - 1, c) \geq 2c - 1$.
\end{proposition}

\begin{proof}
Fix $c \geq 2$ and consider the hypergraph $G = \big([3c - 3], { [3c - 3] \choose 2c - 2} \big)$. This hypergraph is
$(c-1)$-intersecting and all its edges have size $2c-2$.
Consider any coloring of the vertices in $G$ that uses at most $2c-2$ colors.
The most common $c-1$ colors in such a coloring must cover at least
$$
(c - 1) \left\lceil \frac {3c - 3}{2c - 2} \right\rceil = (c - 1)\cdot 2 = 2c - 2
$$
vertices. So one of the edges of $G$ covers vertices with at most $c-1$ distinct
colors and the coloring of $G$ is not $c$-strong.  Thus,
$\chi(c-1,c) \ge \chi(G,c) \ge 2c-1$.
\end{proof}

\begin{proposition}
For any $t \ge c \geq 2$, we have $\chi(t, c) \geq 2(c - 1)$.
\end{proposition}

\begin{proof}
Fix $t \ge c \ge 2$ and consider the hypergraph $G = \big( [(6c-1)t], {[(6c-1)t] \choose 3ct} \big)$.
The hypergraph $G$ is $t$-intersecting and all its edges have size $3ct$.  
Consider any coloring of the vertices in $G$ that uses at most $2c-3$ colors.
The most common $c-1$ colors in such a coloring must cover at least
$$
 \left\lceil \frac{c-1}{2c-3} (6c-1)t \right\rceil =
 \left\lceil 3ct + \frac{2c-1}{2c-3}t \right\rceil > 3ct
$$
vertices. So one of the edges of $G$ covers vertices with at most $c-1$ distinct colors and the
coloring cannot be $c$-strong. Thus, $\chi(t,c) \ge \chi(G,c) \ge 2(c-1)$.
\end{proof}

\section{Probabilistic upper bound}
\label{sec:ub}

For a fixed $0 < p < 1$, the \emph{$p$-biased measure} of a family $\calF$ over $[n]$ is 
$
\mu_p(\calF) := \Pr_{S}[ S \in \calF],
$
where the probability over $S$ is obtained by including each element $i \in [n]$ in $S$
independently with probability $p$.  Such a set $S$ is called a \emph{$p$-biased subset} of $[n]$.
Dinur and Safra~\cite{dinur2005hardness} showed that when $p$
is small enough, $2$-intersecting families have small $p$-biased measure. Friedgut~\cite{friedgut2008measure}
showed how the same result also extends to $t$-intersecting families for every $t > 2$.

\begin{theorem}[Dinur and Safra~\cite{dinur2005hardness}; Friedgut~\cite{friedgut2008measure}]
\label{thm:DS}
Fix $t \ge 1$. Let $\calF$ be a $t$-intersecting family.
For any $p < \frac1{t+1}$, the $p$-biased measure of $\calF$ is bounded by
$
\mu_p(\calF) \le p^t.
$
\end{theorem}

We  obtain upper bounds on the chromatic number of intersecting hypergraphs by applying an 
immediate corollary of Theorem~\ref{thm:DS}.

\begin{corollary}
\label{coro:DS}
Fix $t \ge 1$. Let $\calF$ be a $t$-intersecting family. 
For any $p < \frac1{t+1}$, the probability that a $p$-biased subset of $[n]$ contains
a set $S \in \calF$ is at most $p^t$.
\end{corollary}

\begin{proof}
Fix $\calF$ to be some $t$-intersecting family and define $\calF'$ to be the 
$t$-intersecting family obtained from $\calF$
by adding any set which contains a member of $\calF$. That is,
$\calF' = \{ T' \subseteq [n] \mid \exists T \in \calF \mbox{ s.t. }  T \subseteq T' \}$.
Fix $p < \frac1{t+1}$ and let $S \subseteq [n]$ be a random $p$-biased subset of $[n]$.
The set $S$ contains some set of $\calF$ if and only if $S \in \calF'$.
By Theorem~\ref{thm:DS}, the probability that this 
event occurs is at most $p^t$.
\end{proof}

We use the corollary to argue that when $\ell$ is large enough, a random $\ell$-coloring of a 
$t$-intersecting hypergraph is $c$-strong with positive probability. 

\begin{theorem}
\label{thm:ub}
For every $t \geq c \geq 2$, let $\ell$ be an integer that satisfies $\ell > (c-1)(t+1)$ and
$$
 {\ell \choose c -1} \left( \frac{c-1}{\ell} \right)^t < 1 \ .
$$
Then $\chi(t,c) \le \ell$. In particular, since $\ell = t^t$ satisfies both conditions, 
$\chi(t,c)$ is finite.
\end{theorem}

\begin{proof}
Let $G = ([n], E)$ be a $t$-intersecting hypergraph and let $\ell$ be an integer
that satisfies both conditions of the theorem.
Consider a random coloring of $G$ where each vertex is assigned a color that is chosen
independently and uniformly at random from $[\ell]$.  
Fix $C$ to be a set of $c-1$ colors.  The set $S$ of vertices that receive one of the colors in $C$
is a random subset of $[n]$ where each element is included in $S$ independently with probability
$p = \frac{c-1}{\ell} < \frac1{t+1}$.  By Corollary~\ref{coro:DS}, the probability that $S$ contains
any edge in $E$ is at most $(\frac{c-1}{\ell})^t$.  Applying the union bound over all possible choices
of $c-1$ colors, the probability that some edge in $G$ contains vertices that have at most $c-1$
colors is at most ${\ell \choose c-1}(\frac{c-1}{\ell})^t < 1$. Therefore, 
there exists a $c$-strong coloring of $G$ that requires only $\ell$ colors.
\end{proof}

\begin{remark}
The proof of Theorem~\ref{thm:ub} does more than is required for establishing the value
of $\chi(t,c)$. It shows that when $\ell$ is large enough, a \emph{random} coloring of a
$t$-intersecting hypergraph with $\ell$ colors is $c$-strong with high probability.
\end{remark}

Theorem~\ref{thm:ub} yields different upper bounds for different values of $t$ with respect to a given $c$.
When $t = c$, the best bound obtained by the theorem is exponential in $c$.

\begin{corollary}
For every $c \geq 2$, $\chi(c, c) < \sqrt{c} \cdot e^{c}$.
\end{corollary}

When $t = 2c$, the bound is already much stronger and shows that the chromatic number
$\chi(t,c)$ is polynomial in $c$.

\begin{corollary}\label{coro:two-c}
For every $c \geq 2$ and $t \geq 2c$, $\chi(t, c) < 2c^2$.
\end{corollary}

As $t$ grows beyond $2c+1$, the bound obtained by
Theorem~\ref{thm:ub} does not continue to improve. In fact, it gets much worse.  
Note also that because of the condition $\ell > (c-1)(t+1)$, the theorem
does not yield a sub-quadratic upper bound on $\chi(t,c)$ for any $t \ge c$.

\begin{remark}
The topic of semi-strong coloring of intersecting hypergraphs came up in
the authors' study of property testing of boolean functions~\cite{BWY11}. A common
approach in such testing algorithms is that of implicit learning, where we
randomly partition some domain and identify a small subset of special parts
in the partition. The main obstacle is often to prove that when the
function is far from satisfying the questioned property, no choice of a
small number of special parts would fool the tester. Theorem~\ref{thm:ub},
and particularly Corollary~\ref{coro:two-c}, guarantees that when we randomly partition the
domain into a polynomial number of parts (which are analogous to colors),
with high probability the union of any small number of parts will satisfy
some criteria (such as not completely containing any member of some \emph{bad} 
intersecting family). See~\cite{BWY11} for more details.
\end{remark}

\section*{Acknowledgments}

We thank Noga Alon and Benny Sudakov for helpful discussions and for encouraging us to write  this note.

\bibliographystyle{plain}
\bibliography{coloring_arXiv}

\end{document}